\documentclass{amsart}
\usepackage{hyperref}
\usepackage{mathtools}
\usepackage{enumitem}
\usepackage{graphicx}
\usepackage{tikz}
\usepackage{xcolor}

\newtheorem{theorem}{Theorem}[section]

\newtheorem{corollary}[theorem]{Corollary}

\theoremstyle{definition}

\theoremstyle{remark}

\newcommand{\Mod}{\mathrm{Mod}}
\newcommand{\T}{\mathcal{T}}

\newcommand{\M}{\mathcal{M}}

\newcommand{\diam}{\mathrm{diam}}
\newcommand{\X}{\mathcal{X}}
\newcommand{\Y}{\mathcal{Y}}

\newcommand{\N}{\mathcal{N}}
\newcommand{\A}{\mathcal{A}}
\newcommand{\axis}{\mathrm{axis}}

\definecolor{OliveGreen}{rgb}{0,0.6,0}
\definecolor{BrightPurple}{rgb}{1,0.1,1}

\begin{document}

\title{Growth of Pseudo-Anosov Conjugacy Classes in Teichm\"{u}ller Space}
\author{Jiawei Han}
\address{Department of Mathematics, Vanderbilt University, Nashville, Tennessee 37235}
\email{jiawei.han@vanderbilt.edu}
\maketitle

\begin{abstract}
	Athreya, Bufetov, Eskin and Mirzakhani \cite{athreya2012lattice} have shown the number of mapping class group lattice points intersecting a closed ball of radius $R$ in Teichm\"{u}ller space is asymptotic to $e^{hR}$, where $h$ is the dimension of the Teichm\"{u}ller space. We show for any pseudo-Anosov mapping class $f$, there exists a power $n$, such that the number of lattice points of the $f^n$ conjugacy class intersecting a closed ball of radius $R$ is coarsely asymptotic to $e^{\frac{h}{2}R}$.
\end{abstract}

\section{Introduction}

One can study a group by understanding its ``growth'' in various ways. Consider $G$ acting on a metric space $S$ by isometries, one can measure the number of orbit or lattice points of $G$ in a ball of radius $R$ as $R$ goes to infinity. For example, consider $\mathbb{Z}^3$ acting on $\mathbb{R}^3$ in the standard way, the number of lattice points of $\mathbb{Z}^3$ in a ball of radius $R$ is roughly the volume of this ball, see \cite{lax1982asymptotic} for example. In this paper, we study mapping class groups by understanding the lattice points of pseudo-Anosov conjugacy classes in Teichm\"{u}ller space.

Let $M$ be a compact, negatively curved Riemannnian manifold and let $\tilde{M}$ denote its universal cover. The fundamental group $\pi_1(M)$ acts on $\tilde{M}$ by isometries. Let $B_R(x)$ denote the ball of radius $R$ in $\tilde{M}$ centered at $x$. G.A. Margulis studied the growth rate of any orbit $\pi_1(M) \cdot y$ by intersecting with any metric balls $B_r(x)$.
\begin{theorem}[Margulis \cite{margulis2004some}]
	There is a function $c \colon M \times M \to \mathbb{R}^+$ so that for every $x, y \in \tilde{M}$,
	\begin{align*}
		|\pi_1(M) \cdot y \cap B_R(x)| \sim c(p(x),p(y))e^{hR}
	\end{align*}
	where $h$ equals to the dimension of the manifold, which is the topological entropy of the geodesic flow on the unit tangent bundle of $M$. 
\end{theorem}
The notation $f(R) \sim g(R)$ means $\lim_{R \to \infty} \frac{f(R)}{g(R)} = 1$.

Inspired by this result, Athreya, Bufetov, Eskin and Mirzakhani studied lattice point asymptotics in Teichm\"{u}ller space. Let $S_{g,n}$ denote a closed surface of genus $g$ with $n$ punctures such that $3g-3+n > 0$, and we let $\Mod_{g,n}$ and $(\T_{g,n},d)$ denote the corresponding mapping class group and Teichm\"{u}ller space with Teichm\"{u}ller metric. Then $\Mod_{g,n}$ acts on $\T_{g,n}$ by isometries. We use $\Mod_g, \T_g$ to denote $\Mod_{g,0}, \T_{g,0}$ for simplicity. They showed the orbits of mapping class group have analogous asymptotics.

\begin{theorem}[Athreya, Bufetov, Eskin and Mirzakhani \cite{athreya2012lattice}] \label{ABEM}
	For any $\X, \Y \in \mathcal T_g$, we have
	\begin{align*}
		| \Mod_g \cdot \Y \cap B_R(\X)| \sim e^{hR}
	\end{align*}
\end{theorem}

Note in their original paper, there is a factor of $\Lambda(\X) \Lambda(\Y)$ in front of $ e^{\frac{h}{2}R}$, $\Lambda$ is called the Hubbard-Masur function. Mirzakhani later showed that $\Lambda$ is a constant function, see \cite{10.1007/s11511-015-0129-6}. Moreover, we recall the following result from Parkkonen and Paulin \cite{parkkonen2015hyperbolic} about the lattice point asymptotics for conjugacy classes of $\pi_1(M)$.

\begin{theorem}[Parkkonen, Paulin \cite{parkkonen2015hyperbolic}] \label{conjugacy}
	Let $G$ be a nontrivial conjugacy class of $\pi_1(M)$, for any $x \in \tilde{M}$, we have
	\begin{align*}
		\lim_{R \to \infty} \frac{1}{R}\ln |G \cdot \X \cap B_R(\X) | = \frac{h}{2}.
	\end{align*}
\end{theorem} 

Inspired by this result, we wish to explore the lattice point asymptotics for conjugacy classes of $\Mod_{g,n}$. The Nielsen-Thurston Classification \cite{thurston1988} says every element in $\Mod_g$ is one of the three types: periodic, reducible, or pseudo-Anosov. When $f\in \Mod_{g,n}$ is a Dehn twist, a special kind of reducible element, we prove in \cite{Dehntwistpaper} that the lattice point growth for the conjugacy class of $f$ is "coarsely" asymptotic to $e^{\frac{h}{2}R}$. In this paper, we are interested in pseudo-Anosov elements. Let $PA \subset \Mod_g$ denote the subset of pseudo-Anosov elements. Maher showed pseudo-Anosov elements are generic in the following sense. 

\begin{theorem}[Maher \cite{maher2010asymptotics}] \label{Maher}
	For any $\X, \Y \in \mathcal T_g$, we have
	\begin{align*}
		\frac{|PA \cdot \Y \cap B_R(\X)|}{| \Mod_g \cdot \Y \cap B_R(\X)|} \sim 1.
	\end{align*}
\end{theorem}

The above Theorems \ref{conjugacy}, \ref{Maher}, motivate us to explore the lattice point asymptotics for pseudo-Anosov conjugacy class subgroups. For any mapping class $\phi \in \Mod_{g,n}$, we use $[\phi] = \{f \phi f^{-1} \mid f \in \Mod_{g,n}\}$ to denote its conjugacy class. For simplicity of notation, we denote $\Gamma_R(\X, \Y, \phi) = \left|[\phi] \cdot \Y \cap B_R(\X)\right|$.

Let $C > 0$,  we say $f(R) \stackrel{C}\preceq g(R)$ if for any $\delta > 1$, there exists a $M(\delta)$ such that $\frac{1}{\delta C} \cdot f(R) \le g(R)$ for any $R \ge M(\delta)$. We say $f(R) \stackrel{C}\sim g(R)$ if $f(R) \stackrel{C}\preceq g(R)$ and $g(R) \stackrel{C}\preceq f(R)$, thus $f(R) \stackrel{1}\sim g(R)$ is the same as $f(R) \sim g(R)$. Accordingly, we simply write $\preceq, \sim$ when $C=1$. The main results of this paper are the following.

\begin{theorem}\label{main theorem}
	Fix $S_{g,n}$ and $\epsilon > 0$, there exists a constant $A>0$ such that given any $\epsilon$-thick pseudo-Anosov element $\phi$ with translation distance $\lambda \ge A$ and given any $\X, \Y$ in $\T_{g,n}$, there exists a corresponding $G(\X, \Y, \phi)$ such that
	\begin{align*}
		\Gamma_R(\X, \Y, \phi) \stackrel{G(\X, \Y, \phi)}{\sim} e^{\frac{h}{2}R}.
	\end{align*}
\end{theorem}

\begin{corollary}\label{main corollary}
	Fix $S_{g,n}$, given any pseudo-Anosov element $\phi$ and given any $\X, \Y$ in $\T_{g,n}$. There exists a power $N$ depending on $\phi$ such that for any $k \ge N$, there is a corresponding $G(\X, \Y, \phi, k)$ so that the following holds:
	\begin{align*}
		\Gamma_R(\X, \Y, \phi^k) \stackrel{G(\X,\Y, \phi, k)}{\sim} e^{\frac{h}{2}R}.
	\end{align*}
\end{corollary}

In parallel with the Theorem \ref{conjugacy} above, we note the above Theorem \ref{main theorem} and Corollary \ref{main corollary} imply the following.
\begin{corollary}\label{second corollary}
Fix $S_{g,n}$, given any pseudo-Anosov element $\phi$ and given any $\X, \Y$ in $\T_{g,n}$, for all sufficiently large $k$ we have
\begin{align*}
    \lim_{R \to \infty} \frac{1}{R}\ln \Gamma_R(\X, \Y, \phi^k) = \frac{h}{2}.
\end{align*}
\end{corollary}

These results again indicate the similarity of Teichm\"{u}ller spaces and hyperbolic spaces in terms of lattice point asymptotics.

\subsection{Acknowledgments}
I would like to thank my advisor, Spencer Dowdall, for suggesting this project, for his guidance and consistent support throughout.

\section{Background}

Let $\text{Homeo}^+_{g,n}$ denote the group of all the orientation-preserving homeomorphisms of $S_{g,n}$ preserving the set of punctures, and let $\text{Homeo}^0_{g,n}$ denote the connected component of the identity. The mapping class group of $S_{g,n}$ is defined to be the group of isotopy classes of orientation-preserving homeomorphisms:
\begin{align*}
	\Mod_{g,n} = \text{Homeo}^+_{g,n}/ \text{Homeo}^0_{g,n} = \text{Homeo}^+_{g,n}/\ \text{isotopy} 
\end{align*}

A hyperbolic structure $\X$ on $S_{g,n}$ is a pair $(X, \phi)$ where $\phi\colon S_{g,n} \to X$ is a homeomorphism and $X$ is a hyperbolic surface. We say two hyperbolic structures $\X = (X, \phi), \Y = (Y, \psi)$ are isotopic if there is an isometry $I\colon X \to Y$ isotopic to $\psi \circ \phi^{-1}$.   The Teichm\"{u}ller space $\T_{g,n}$ is the set of hyperbolic structures on $S_{g,n}$ modulo isotopy. We let $\X = (X, \phi), \Y = (Y, \psi)$ denote elements in $\T_{g,n}$. Given any $\X, \Y \in \T_{g,n}$, the Teichm\"{u}ller distance between them is defined to be
\begin{align*}
	d_\T(\X, \Y) = \frac{1}{2} \inf_{f \sim \phi \circ \psi^{-1}}\log(K_f)
\end{align*}
where the infimum is over all quasi-conformal homeomorphisms $f$ isotopic to $\phi \circ \psi^{-1}$ and $K_f$ is the quasi-conformal dilatation of $f$. Equipped with the Teichm\"{u}ller metric, the Teichm\"{u}ller space is a complete, unique geodesic metric space.

Given any $\X = (X, \phi) \in \T_{g,n}$ and given any isotopy class $\gamma$ of nontrivial simple closed curves on $S_{g,n}$, there exists a unique geodesic in the free homotopy class of $\phi(\gamma)$ on $X$. We define $\ell_{X}\left(\phi(\gamma)\right)$ to the length of this unique geodesic and define $\ell_\X(\gamma) = \ell_{X}\left(\phi(\gamma)\right)$. A pants decomposition of the surface $S_{g,n}$ is a collection of pairwise disjoint non-trivial simple closed curves $\gamma_1, \cdots, \gamma_{3g-3+n}$ on $S_{g,n}$, together they decompose the surface $S_{g,n}$ into $2g+n-2$ pairs of pants. Using pants decomposition and by introducing Fenchel-Nielsen coordinates, Fricke \cite{FrickeKlein} showed that $\T_{g,n}$ is homeomorphic to $\mathbb R^{6g+2n-6}$.

The mapping class group acts isometrically on $\T_{g,n}$ by changing the marking $(f, (X,\phi)) \mapsto (X,\phi \circ f^{-1})$. This action is properly discontinuous but not cocompact. The quotient $\M_{g,n} = \T_{g,n} / \Mod_{g,n}$ is called the moduli space, and it is a non-compact orbifold parameterizing hyperbolic surfaces homeomorphic to $S_{g,n}$.

Given any $\epsilon > 0$, the $\epsilon$-thick part of Teichm\"{u}ller space is defined to be
\begin{align*}
	\T^\epsilon_{g,n} = \{\X \in \T_{g,n} \mid \ell_\X(\alpha) \ge \epsilon \text{\ for any simple closed curve\ } \alpha \text{\ on\ } S_{g,n}\}
\end{align*}
and consequently the $\epsilon$-thick part of moduli space is $\M^\epsilon_{g,n} = \T_{g,n}^\epsilon / \Mod_{g,n}$. The Mumford compactness criterion \cite{10.2307/2037802} says $\M^\epsilon_{g,n}$ is compact for any $\epsilon >0$.

Similar to hyperbolic isometrics acting on hyperbolic space, each pseudo-Anosov element $\phi \in \Mod_{g,n}$ acts on $\T_{g,n}$ by translating along its corresponding bi-infinite geodesic axis, denoted as $\axis(\phi)$ with translation distance denoted as $\lambda(\phi)$. Moreover, we say a pseudo-Anosov element $\phi \in \Mod_{g,n}$ is called $\epsilon$-thick if its axis $\axis(\phi) \subset \T_{g,n}^\epsilon$.

For any $r>0$ and for every closed set $W \subset \T_{g,n}$, denote $\N_r(W)$ the $r$-neighborhood of $W$. For every closed set $C \subset \T_{g,n}$, the closest point projection map is defined as follows
\begin{align*}
\pi_C(x) = \{y \in C \mid d(x,y) = d(x, C) = \inf_{z \in C} d(x,z) \}.
\end{align*}

As one of the early works exploring negative curvature in Teichm\"{u}ller space, the result below from Minsky \cite{minsky1994quasiprojections} says that $\epsilon$-thick geodesics in Teichm\"{u}ller space satisfy the strongly contracting property.

\begin{theorem}[Minsky \cite{minsky1994quasiprojections}] \label{MinskyProjection}
There exists a constant $A > 0$ depending on $\epsilon, \chi(S)$ such that if $\mathcal L$ is an $\epsilon$-thick geodesic in $\T_{g,n}$ and $d(\X, \mathcal L) > A$, then we have
\begin{align*}
\diam\left(\pi_{\mathcal L}\left(\mathcal N_{d\left(\X, L\right)-A}\left(\X\right)\right)\right) \le A
\end{align*}
for any $\X \in \T_{g,n}$.
\end{theorem}

For $\mathcal L$ a geodesic in $\T_{g,n}$, we let $d_\pi^{\mathcal L}(C, W) = \diam\left(\pi_{\mathcal L}\left(C\right) \cup \pi_{\mathcal L}\left(W\right)\right)$. We can pick the constant $A$ in Theorem \ref{MinskyProjection} in a way so that the following holds.

\begin{corollary}[Arzhantseva, Cashen, and Tao, \cite{Arzhantseva_2015}] \label{MinskyCorollary}
Let $\mathcal L$ be an $\epsilon$-thick geodesic in $\T_{g,n}$ and let $\X, \Y \in \T_{g,n}$ be such that $d_\pi^{\mathcal L}(\X, \Y) > A$, then
\begin{align*}
d\left(\X,\Y\right) \ge d\left(\X, \pi_{\mathcal L}\left(\X\right)\right)+d_\pi^{\mathcal L}(\X, \Y)+d\left(\pi_{\mathcal L}\left(\Y\right), \Y\right) - A.
\end{align*}
Moreover, if $\Y$ happens to be on the geodesic $\mathcal L$, then $\pi_{\mathcal L}(\Y) = \{\Y\}$ and
\begin{align*}
d(\X,\Y) \ge d\left(\X, \pi_{\mathcal L}\left(\X\right)\right)+d\left(\pi_{\mathcal L}\left(\X\right), \Y\right) - A.
\end{align*}
\end{corollary}

For any pseudo-Anosov element $\phi \in \Mod_{g,n}$,  we denote $\pi_{\axis(\phi)}$ as $\pi_\phi$. Since $\phi$ acts by translation along its axis, it commutes with the projection map $\pi_\phi$. That is, for any $\X \in \T_{g,n}$, we have $\pi_{\phi}(\phi(\X)) = \phi(\pi_{\phi}(\X))$.

By using Theorem \ref{MinskyProjection} and Corollary \ref{MinskyCorollary}, one can show if an $\epsilon$-thick pseudo-Anosov element $\psi$ has sufficiently large translation length, then the distance it translates a point is roughly twice the distance from the point to the axis. See Figure \ref{pic1} for an illustration.

\begin{figure}[h]
	\begin{center}
		\begin{tikzpicture}
			\node[anchor=south west,inner sep=0] at (0,0) {\includegraphics[scale=0.35]{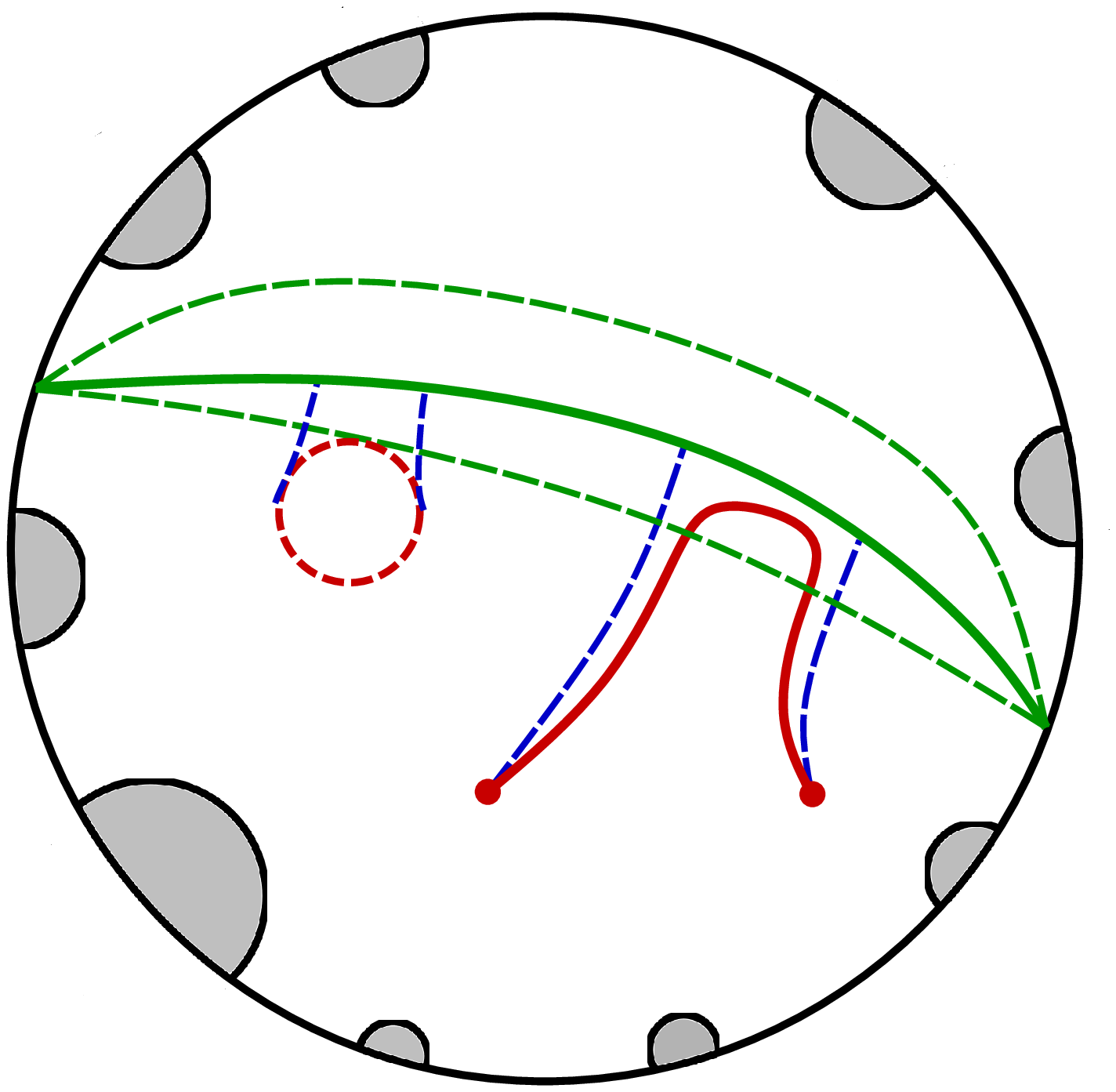}};
			\node at (1.85, 2.6){\textcolor{red}{$B$}};
			\node at (5.7, 1.5){\textcolor{OliveGreen}{$\axis(\psi)$}};
			\node at (2.5, 1){\textcolor{red}{$\X$}};
			\node at (3.9, 1){\textcolor{red}{$\psi(\X)$}};
		\end{tikzpicture}	
	\end{center}
	\caption{Shaded area are $\epsilon$-thin parts. Given a $\epsilon$-thick pseudo-Anosov element $\psi$ with $\lambda(\psi) > A$, the diameter of projection of any balls like $B$ to $\axis(\psi)$ is bounded by $A$, see Theorem \ref{MinskyProjection}. The geodesic from $\X$ to $\psi(\X)$ fellow travels $\axis(\psi)$, see Corollary \ref{thick pA distance}.}
	\label{pic1}
\end{figure}

\begin{corollary}\label{thick pA distance}
Let $\phi$ be a $\epsilon$-thick pseudo-Anosov element with translation distance $\lambda(\phi) > A$. Then for any $\X \in \T_{g,n}$ and for any $\psi \in [\phi]$, we have
\begin{align*}
 2d(\X, \pi_{\psi}(\X)) + \lambda(\phi) - A \le d(\X, \psi(\X)) \le 2d(\X, \pi_{\psi}(\X)) + \lambda(\phi) + 2A.
\end{align*}
\end{corollary}
\begin{proof}
Since translation distance is invariant under conjugation, $\lambda(\psi) = \lambda(\phi) > A$ for any $\psi \in [\phi]$. Thus we have
\begin{align*}
d_\pi^\psi(\X, \psi(\X)) = \diam(\pi_{\psi}(\X) \cup \pi_{\psi}(\psi(\X))) = \diam (\pi_{\psi}(\X) \cup \psi(\pi_{\psi}(\X)))
\end{align*}
where $\lambda(\phi) \le \diam \left(\pi_{\psi}\left(\X\right) \cup \psi \left(\pi_{\psi}\left(\X\right)\right)\right) \le \lambda(\phi)+2A$. Take any $\X \in \T_{g,n}$, by the triangle inequality, we have
\begin{align*}
d(\X,\psi(\X)) \le &d(\X, \pi_{\psi}(\X))+ d_\pi^\psi(\X, \psi(\X))+d(\psi(\X), \pi_{\psi}(\psi(\X))) \\
\le& 2d(\X, \pi_{\psi}(\X))+ \lambda(\phi) +2A.
\end{align*}
Meanwhile we can apply the previous Corollary \ref{MinskyCorollary} and get
\begin{align*}
d(\X,\psi(\X)) \ge &d(\X, \pi_{\psi}(\X))+d_\pi^\psi(\X, \psi(\X))+d(\psi(\X), \pi_{\psi}(\psi(\X)))-A \\
\ge& 2d(\X, \pi_{\psi}(\X))+ \lambda(\phi)-A.
\end{align*}
The result follows.
\end{proof}

\section{Proof of the main theorem}

By Theorem \ref{ABEM}, for any $\X \in \T_{g,n}$, we have
\begin{align*}
	\left|\Mod_{g,n} \cdot \X \cap B_r(\X)\right| \sim e^{hr}.
\end{align*}

For any $r > 0$, define the set
\begin{align*}
	\Omega_r(\X) = \{f \in \Mod_{g,n} \mid d(\X, f\X) \le r\}
\end{align*}
and denote $N$ the maximal order of point stabilizer subgroups in $\Mod_{g,n}$ \cite{kerckhoff1980}. It follows that
\begin{align*}
	\left|\Mod_{g,n} \cdot \X \cap B_r(\X)\right| \le &\left|\Omega_r(\X)\right| \le N\cdot \left|\Mod_{g,n} \cdot \X \cap B_r(\X)\right|, \\
	e^{hr} \preceq &\left|\Omega_r(\X)\right| \preceq N \cdot e^{hr}.
\end{align*}
Moreover, given any $\phi \in \Mod_{g,n}$, we have
\begin{align*}
	\Gamma_r(\X, \Y, \phi) \le \left| [\phi] \cap \Omega_r(\X) \right| \le N \cdot \Gamma_r(\X, \Y, \phi).
\end{align*}
Combining things together, we have
\begin{align}
	\frac{1}{N} \cdot \left|[\phi] \cap \Omega_{r}(\X)\right| \le \Gamma_r(\X, \Y, \phi) \le \left|[\phi] \cap \Omega_{r}(\X)\right|. \label{upper and lower bound 1}
\end{align}

We first prove a simplified version of the main theorem.

\begin{theorem}\label{simple version theorem}
	For any $S_{g,n}$ and $\epsilon > 0$, there exists a constant $A>0$ such that given any $\epsilon$-thick pseudo-Anosov element $\phi$ with translation distance $\lambda \ge A$ and given any $\X \in \axis(\phi)$, there exists a corresponding constant $G(\X, \phi) > 0$ such that
	\begin{align*}
		\Gamma_R(\X, \X, \phi) \stackrel{G(\X, \phi)}{\sim} e^{\frac{h}{2}R}
	\end{align*}
\end{theorem}
\begin{proof}
	Given $\phi, \X$ satisfying the assumptions. For any $R$, define
	\begin{align*}
		&P^+_R = \left\{\psi \in [\phi] \mid d\left(\X, \pi_{\psi}(\X)\right) \le \frac{R+A-\lambda}{2}\right\}, \\
		&P^-_R = \left\{\psi \in [\phi] \mid d\left(\X, \pi_{\psi}(\X)\right) \le \frac{R-2A-\lambda}{2}\right\}.
	\end{align*}
	Denote $\Omega_{r}(\X)$ as $\Omega(r)$ for simplicity, by Corollary \ref{thick pA distance} we have
	\begin{align}
		P^-_R \subset [\phi] \cap \Omega(R) \subset P^+_R. \label{upper and lower bound 2}
	\end{align}
	
	We now work towards obtaining an upper bound for $|P^+_R|$. Take any $\psi \in P^+_R$, there exists a $f \in \Mod_{g,n}$ such that $\psi = f \phi  f^{-1}$. Since $\X \in \axis(\phi)$, $f(\X)$ therefore lies on the $\axis(\psi)$. In particular, this means there exists a $k \in \mathbb{Z}$ such that
	\begin{align*}
		&d\left(\psi^k \circ f(\X), \pi_{\psi}(\X)\right) \le \frac{\lambda}{2}, \\
		&d\left(\psi^k \circ f(\X), \X\right) \le d\left(\psi^k \circ f(\X), \pi_{\psi}(\X)\right) + d\left(\X, \pi_{\psi}(\X)\right) \le \frac{R+A}{2}.
	\end{align*}
	See Figure \ref{pic2} for an example.
	
	\begin{figure}[h]
	\begin{center}
		\begin{tikzpicture}
			\node[anchor=south west,inner sep=0] at (0,0) {\includegraphics[scale=0.35]{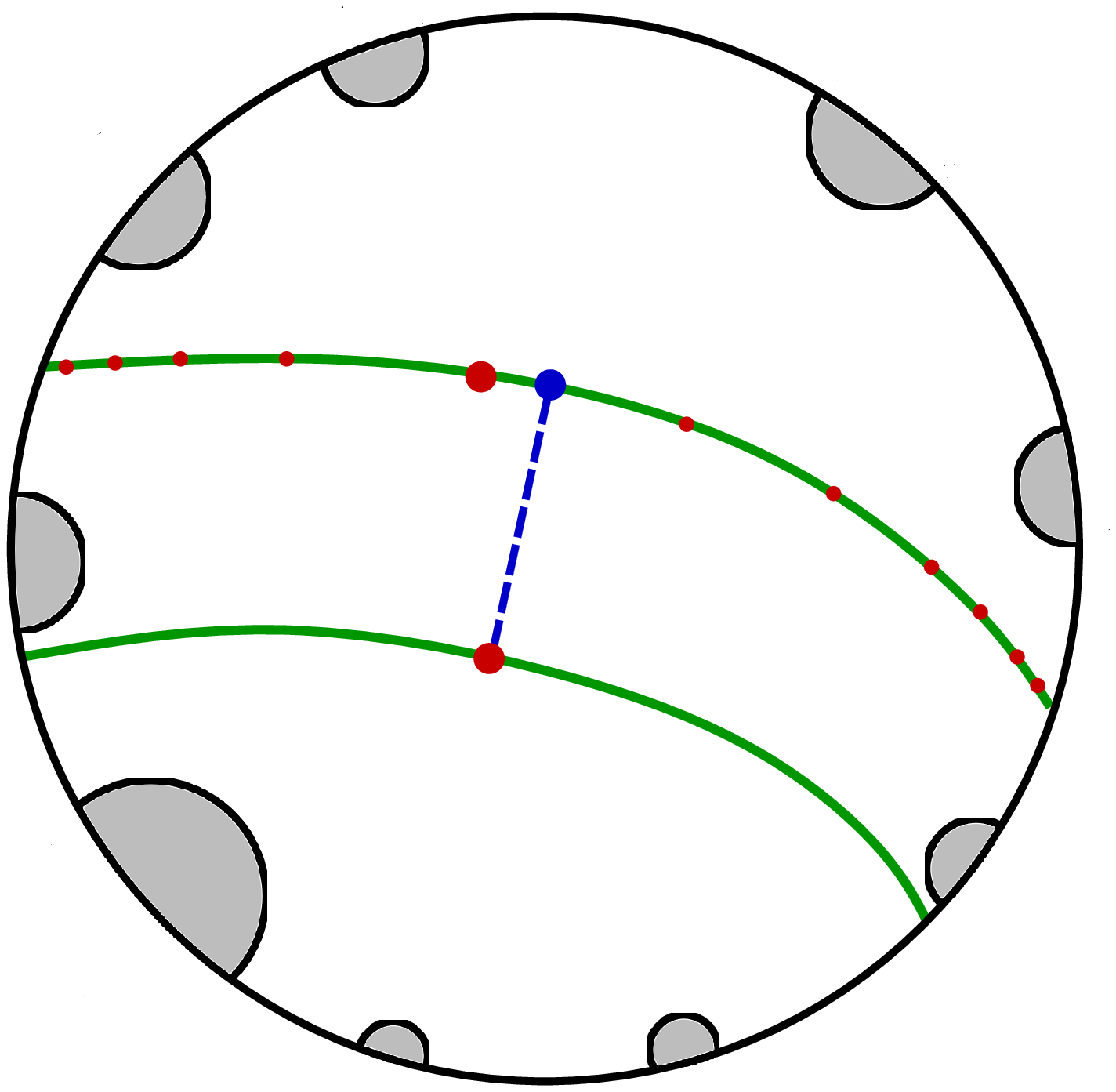}};
			\node at (3.2, 3.45){\textcolor{blue}{$\pi_{\psi}(\X)$}};
			\node at (5.8, 1.5){\textcolor{OliveGreen}{$\axis(\psi)$}};
			\node at (-0.2, 1.5){\textcolor{OliveGreen}{$\axis(\phi)$}};
			\node at (2.2, 1.7){\textcolor{red}{$\X$}};
			\node at (4.3, 2.2){\textcolor{red}{$x_0$}};
			\node at (3.9, 2.5){\textcolor{red}{$x_1$}};
			\node at (3.3, 2.8){\textcolor{red}{$x_2$}};
			\node at (2.3, 3){\textcolor{red}{$x_3$}};
			\node at (1.65, 3.05){\textcolor{red}{$x_4$}};
			\node at (1.1, 3.05){\textcolor{red}{$x_5$}};
		\end{tikzpicture}	
	\end{center}
	\caption{Each $x_i$ denotes $\psi^i \circ f(x)$ and distance between any two adjacent $x_i$ is $\lambda$. The injective map maps $\X$ to $x_3$ since $x_3$ is the closest point to $\pi_{\psi}(\X)$ in $\{x_i\}_{i \in \mathbb{Z}}$.}
	\label{pic2}
\end{figure}

	We claim one can define an injective map from $P^+_R \to \Omega(\frac{R+A}{2})$ by sending $\psi$ to $\psi^k f$. Indeed, if there is any another $\eta \in P^+_R, \eta \neq \psi$, $\eta = h \phi h^{-1}$ for some $h \in \Mod_{g,n}$, then $h(\X) \in \axis(\eta)$ and there exists a $m \in \mathbb{Z}$ such that
	\begin{align*}
		&d(\eta^m \circ h(\X), \pi_{\eta}(\X)) \le \frac{\lambda}{2}, \\
		&d(\eta^m \circ h(\X), \X) \le \frac{R+A}{2}.
	\end{align*}
	We claim in this case $\psi^k f \neq \eta^m h$. Indeed, suppose they are equal, then
	\begin{align*}
		\psi = \psi^k \psi \psi^{-k}  =  \psi^k f \phi f^{-1} \psi^{-k} = \eta^m h \phi h^{-1} \eta^{-m} = \eta^m \eta \eta^{-m} = \eta.
	\end{align*}
	However, this contradicts $\psi \neq \eta$. This means for $R$ large, we can inject $P_R^+$ into $\Omega(\frac{R+A}{2})$, so that
	\begin{align}
		&\left|P^+_R \right| \le \left|\Omega\left(\frac{R+A}{2}\right)\right| \preceq e^{\frac{hA}{2}} \cdot e^{\frac{hR}{2}}. \label{upper bound 1}
	\end{align}
	
	To obtain the lower bound for $|P^-_R|$, we define $\A_R =\left\{\axis(\psi) \mid \psi \in P_R^- \right\}$. This gives us a surjective map $F: P^-_R \to \A_R, \psi \mapsto \axis(\psi)$, and each $\Theta \in \A_R$ has the form $\Theta = \axis( f \phi f^{-1} )$ for some $f \in \Omega(\frac{R-2A}{2})$. For any $L < \frac{R-2A-\lambda}{2}$, we define $\A_R^L =\left\{\Theta \in \A_R \mid d\left(\X, \pi_{\Theta}(\X)\right) > \frac{R-2A-\lambda}{2} - L \right\}$ so that $\A_R^L\subset \A_R$. For each $\Theta \in \A_R $, we denote $H(\Theta) = \{f \in \Omega(\frac{R-2A}{2}) \mid \axis(f\phi f^{-1}) = \Theta \}$, which is a subset of $\Omega(\frac{R-2A}{2})$.
	
	By Corollary \ref{MinskyCorollary}, for any $\Theta \in \A_R^L$, there are at most $\frac{2(L+A)}{\lambda}+2$ many $f \in H(\Theta)$ satisfying $\axis(f\phi f^{-1}) = \Theta$ since $d(\X,\pi_{\Theta}(\X)) \in \left(\frac{R-2A-\lambda}{2}-L,\frac{R-2A-\lambda}{2}\right]$. In the example of Figure \ref{pic3}, there are six such $f$ for this $\Theta$. This means
	\begin{align}
		\left|\A_R^L \right| \ge \frac{\lambda}{2(L+A+\lambda)} \cdot \sum_{\Theta \in \A_R^L} \left|  H\left(\Theta\right)\right|. \label{lower bound mess 1}
	\end{align}
	For any element $f \in \Omega(\frac{R-2A-\lambda}{2})$, let's denote $\Theta_f = \axis(f \phi f^{-1})$, then each $f$ is exactly one of the following types.
	\begin{enumerate}[label=(\alph*)]
		\item $\Theta_f$ never enters $B_{\frac{R-2A-\lambda}{2}-L}(\X)$. 
		\item $\Theta_f$ enters $B_{\frac{R-2A-\lambda}{2}-L}(\X)$ and $d(\X, f(\X)) \le {\frac{R-2A-\lambda}{2}-L}$.
		\item $\Theta_f$ enters $B_{\frac{R-2A-\lambda}{2}-L}(\X)$ and $d(\X, f(\X)) > {\frac{R-2A-\lambda}{2}-L}$.
	\end{enumerate}

	\begin{figure}[h]
	\begin{center}
		\begin{tikzpicture}
			\node[anchor=south west,inner sep=0] at (0,0) {\includegraphics[scale=0.4]{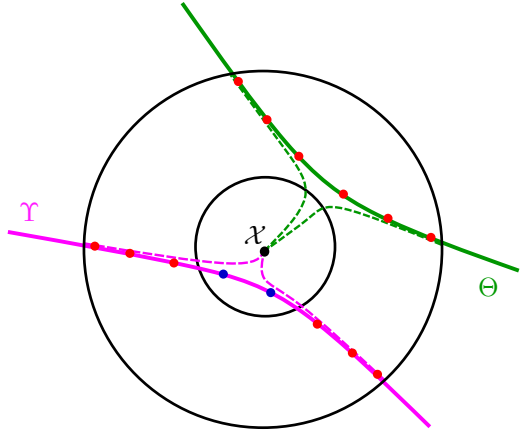}};
			\node at (6.4, 1.9){\textcolor{OliveGreen}{$\Theta$}};
			\node at (0.3, 2.9){\textcolor{BrightPurple}{$\Upsilon$}};
			\node at (3.3, 2.6){$\X$};
		\end{tikzpicture}	
	\end{center}
	\caption{$\Theta$ is of type (a) and $\Upsilon$ is of type (c). The lengths of $\Theta$ and $\Upsilon$ intersecting $B_{\frac{R-2A-\lambda}{2}}$ can be approximated by Corollary \ref{MinskyCorollary}, which showed as the dotted geodesic segments.}
	\label{pic3}
	\end{figure}

	The union of type (a) elements is $\bigsqcup_{\Theta \in \A_R^L} H(\Theta)$, and the union of type (b) elements are $\Omega\left(\frac{R-2A-\lambda}{2}-L\right) \subset \Omega\left(\frac{R-2A}{2}-L\right)$. By Corollary \ref{MinskyCorollary}, we notice there are at most $\frac{2(L+A)}{\lambda}$ many type (c) elements can share the same axis, and the numbers of axes going through $B_{\frac{R-2A-\lambda}{2}-L}(\X)$ is bounded by $|\Omega\left(\frac{R-2A}{2}-L\right)|$. In the example of Figure \ref{pic3}, there are six $f$ satisfying type (c) conditions sharing the axis $\Upsilon$. Notice there are two $f$ realize $\Upsilon = \Theta_f$ but not satisfy the type (c) assumption. Since type (a), (b), (c) elements compose $\Omega(\frac{R-2A-\lambda}{2})$, we have
	\begin{align*}
		\sum_{\Theta \in \A_R^L} \left|  H\left(\Theta\right) \right| \ge \left|\Omega\left(\frac{R-2A-\lambda}{2}\right)\right|- \left(1+ \frac{2(L+A)}{\lambda} \right) \cdot \left|\Omega\left(\frac{R-2A}{2}-L\right)\right|.
	\end{align*}
	Moreover, we let $L$ be a constant satisfy $e^{hL} > 2 \cdot e^{h\frac{\lambda}{2}}  \cdot N \left(1+\frac{2(L+A)}{\lambda}\right)$, then
	\begin{align}
			\sum_{\Theta \in \A_R^L} \left|  H\left(\Theta\right) \right| &\succeq e^{\frac{h(R-2A-\lambda)}{2}} - \left(1+ \frac{2(L+A)}{\lambda} \right) \cdot N \cdot e^{\frac{h(R-2A)}{2}-hL} \label{lower bound mess 2} \\ 
			&\succeq e^{\frac{h}{2}R} \cdot e^{-hA} \cdot \left( \frac{1}{e^{h\frac{\lambda}{2}}} - \frac{N \cdot \left(1+ \frac{2(L+A)}{\lambda} \right)}{e^{hL}}  \right) \succeq e^{\frac{h}{2}R} \cdot \frac{1}{2e^{h(\frac{\lambda}{2}+A)}}, \nonumber
	\end{align}
	and this lower bound is nontrivial.
	
	Thus, to construct the lower bound for $|P_R^-|$, we let $L$ be a constant satisfy  $e^{hL} > 2 \cdot e^{h\frac{\lambda}{2}}  \cdot N \left(1+\frac{2(L+A)}{\lambda}\right)$. Apply formulas (\ref{lower bound mess 1}) (\ref{lower bound mess 2}) from above, for $R$ large we have
	\begin{align}
		\left|P^-_R\right| &\ge |\A_R| \ge \left|\A_R^L\right| \ge \frac{\lambda}{2(L+A+\lambda)} \cdot \sum_{\Theta \in \A_R^L} \left|  H\left(\Theta\right)\right| \label{lower bound 1} \\
		&\succeq e^{\frac{h}{2}R} \cdot \frac{\lambda}{2(L+A+\lambda)e^{hA}} \cdot \frac{1}{2e^{h(\frac{\lambda}{2}+A)}}. \nonumber
	\end{align}
	
	Finally, combining formulas (\ref{upper and lower bound 1}), (\ref{upper and lower bound 2}), (\ref{lower bound 1}) we have
	\begin{align*}
		\left|[\phi] \cdot \X \cap B_R(\X)\right| \ge \frac{1}{N} \cdot \left|[\phi] \cap \Omega(R)\right| \ge \frac{1}{N} \cdot \left|P^-_R\right| \succeq G_L(\X, \phi) \cdot e^{\frac{h}{2}R}
	\end{align*}
	where
	\begin{align*}
		G_L(\X, \phi) = \frac{\lambda}{2N(L+A+\lambda)e^{hA}} \cdot \frac{1}{2e^{h(\frac{\lambda}{2}+A)}}.
	\end{align*}
	And combining formulas (\ref{upper and lower bound 1}), (\ref{upper and lower bound 2}), (\ref{upper bound 1}) we have
	\begin{align*}
		\left|[\phi] \cdot \X \cap B_R(\X)\right| &\le \left|[\phi] \cap \Omega(R)\right| \le P^+_{R} \preceq G_U(\X, \phi) \cdot e^{\frac{h}{2}R}
	\end{align*}
	where
	\begin{align*}
		G_U(\X, \phi) =	Ne^{\frac{hA}{2}}.
	\end{align*}
	Recall $f(R) \stackrel{A}\preceq g(R)$ is the same as $f(R) \stackrel{1}\preceq Ag(R)$. Thus we have
	\begin{align*}
		e^{\frac{h}{2}R} \stackrel{G^{-1}_L(\X, \phi)}\preceq \left|[\phi] \cdot \X \cap B_R(\X)\right| \stackrel{G_U(\X, \phi)}\preceq  e^{\frac{h}{2}R}
	\end{align*}
	This means by setting
	\begin{align*}
		G(\X, \phi) = \max \{G_L^{-1}(\X, \phi), G_U(\X, \phi)\}
	\end{align*}
	we obtain the desired result.
\end{proof}

Now we are ready to prove the general case.

\begin{proof}[Proof of Theorem \ref{main theorem}]
	Take any $\X, \Y \in \T_{g,n}$, and let $D$ be the maximum between $d(\X, \pi_{\phi}(\X))$ and $d(\pi_{\phi}(\X), \Y)\}$. We then have
	\begin{align*}
	&\left|[\phi] \cdot \Y \cap B_{R}(\X)\right| \ge \left|[\phi] \cdot \pi_{\phi}(\X) \cap B_{R-D}(\X)\right| \ge \left|[\phi] \cdot \pi_{\phi}(\X) \cap B_{R-2D}(\pi_{\phi}(\X))\right|, \\
	&\left|[\phi] \cdot \Y \cap B_{R}(\X)\right| \le \left|[\phi] \cdot \pi_{\phi}(\X)  \cap B_{R+D}(\X)\right| \le \left|[\phi] \cdot \pi_{\phi}(\X) \cap B_{R+2D}(\pi_{\phi}(\X))\right|.
	\end{align*}
	 By applying these inequalities and by applying Theorem \ref{simple version theorem} to $\phi$ and $\pi_{\phi}(\X)$, without loss of generality, we get the desired result by setting $G(\X, \Y, \phi) =  G(\pi_\phi(\X), \phi) \cdot e^{hD}$.
\end{proof}

\begin{proof}[Proof of Corollary \ref{main corollary}]
Given $\phi$, we pick $\epsilon$ so that $\axis(\phi)$ is in $\T_{g,k}^\epsilon$. Since $\lambda(\phi^k) = k \cdot \lambda(\phi)$ for any pseudo-Anosov element $\phi$, there exists a $N(\phi)$ such that $\lambda(\phi^k) \ge A$ for any $k \ge N(\phi)$. We now can apply Theorem \ref{main theorem}, and the corresponding error constant $G$ depends on $\X, \Y, \phi, k$.
\end{proof}

\begin{proof}[Proof of Corollary \ref{second corollary}]
Assuming the conditions, we can apply the Corollary \ref{main corollary}. This means for any $k \ge N$ and for any $\delta > 1$, there exists a $M(\delta)$ such that
\begin{align*}
   \frac{1}{\delta G(\X,\Y, \phi, k)} \cdot e^{\frac{h}{2}R} \le \Gamma_R(\X, \Y, \phi^k) \le \delta G(\X,\Y, \phi, k) \cdot e^{\frac{h}{2}R}
\end{align*}
for any $R \ge M(\delta)$. Let $\epsilon > 0$, one can pick $\delta > 0$ and pick $M(\epsilon) \ge M(\delta)$ so that
\begin{align*}
    &\delta G(\X,\Y, \phi, k) \le e^{\epsilon \frac{h}{2}R}, \\
    &e^{-\epsilon \frac{h}{2}R} \le \frac{1}{\delta G(\X,\Y, \phi, k)},
\end{align*}
for any $R \ge M(\epsilon)$. This implies for any $\epsilon >0$, we have
\begin{align*}
    &e^{(1-\epsilon)\frac{h}{2}R} \le \Gamma_R(\X, \Y, \phi^k) \le e^{(1+\epsilon)\frac{h}{2}R}, \\
    &(1-\epsilon)\frac{h}{2}R \le \ln \Gamma_R(\X, \Y, \phi^k) \le (1+\epsilon)\frac{h}{2}R, \\
    &(1-\epsilon)\frac{h}{2} \le \frac{1}{R}\ln \Gamma_R(\X, \Y, \phi^k) \le (1+\epsilon)\frac{h}{2},
\end{align*}
whenever $R \ge M(\epsilon)$. That is,
\begin{align*}
    \lim_{R \to \infty} \frac{1}{R}\ln \Gamma_R(\X, \Y, \phi^k) = \frac{h}{2}.
\end{align*}
This finishes the proof.
\end{proof}

\bibliography{references}

\providecommand{\bysame}{\leavevmode\hbox to3em{\hrulefill}\thinspace}
\providecommand{\MR}{\relax\ifhmode\unskip\space\fi MR }
\providecommand{\MRhref}[2]{%
  \href{http://www.ams.org/mathscinet-getitem?mr=#1}{#2}
}
\providecommand{\href}[2]{#2}
\begin{thebibliography}{10}

\bibitem{Arzhantseva_2015}
Goulnara~N. Arzhantseva, Christopher~H. Cashen, and Jing Tao, \emph{Growth
  tight actions}, Pacific J. Math. \textbf{278} (2015), no.~1, 1--49.
  \MR{3404665}

\bibitem{athreya2012lattice}
Jayadev Athreya, Alexander Bufetov, Alex Eskin, and Maryam Mirzakhani,
  \emph{Lattice point asymptotics and volume growth on {T}eichm\"{u}ller
  space}, Duke Math. J. \textbf{161} (2012), no.~6, 1055--1111. \MR{2913101}

\bibitem{10.1007/s11511-015-0129-6}
David Dumas, \emph{Skinning maps are finite-to-one}, Acta Math. \textbf{215}
  (2015), no.~1, 55--126. \MR{3413977}

\bibitem{FrickeKlein}
Robert Fricke and Felix Klein, \emph{Vorlesungen \"{u}ber die {T}heorie der
  automorphen {F}unktionen. {B}and 1: {D}ie gruppentheoretischen {G}rundlagen.
  {B}and {II}: {D}ie funktionentheoretischen {A}usf\"{u}hrungen und die
  {A}ndwendungen}, Bibliotheca Mathematica Teubneriana, B\"{a}nde 3, vol.~4,
  Johnson Reprint Corp., New York; B. G. Teubner Verlagsgesellschaft, Stuttg
  art, 1965. \MR{0183872}

\bibitem{Dehntwistpaper}
Jiawei Han, \emph{Growth rate of {D}ehn twist lattice points in
  {T}eichm\"{u}ller space}, forthcoming.

\bibitem{kerckhoff1980}
Steven~P. Kerckhoff, \emph{The {N}ielsen realization problem}, Ann. of Math.
  (2) \textbf{117} (1983), no.~2, 235--265. \MR{690845}

\bibitem{lax1982asymptotic}
Peter~D. Lax and Ralph~S. Phillips, \emph{The asymptotic distribution of
  lattice points in {E}uclidean and non-{E}uclidean spaces}, J. Functional
  Analysis \textbf{46} (1982), no.~3, 280--350. \MR{661875}

\bibitem{maher2010asymptotics}
Joseph Maher, \emph{Asymptotics for pseudo-{A}nosov elements in
  {T}eichm\"{u}ller lattices}, Geom. Funct. Anal. \textbf{20} (2010), no.~2,
  527--544. \MR{2671285}

\bibitem{margulis2004some}
Grigoriy~A. Margulis, \emph{On some aspects of the theory of {A}nosov systems},
  Springer Monographs in Mathematics, Springer-Verlag, Berlin, 2004, With a
  survey by Richard Sharp: Periodic orbits of hyperbolic flows, Translated from
  the Russian by Valentina Vladimirovna Szulikowska. \MR{2035655}

\bibitem{minsky1994quasiprojections}
Yair~N. Minsky, \emph{Quasi-projections in {T}eichm\"{u}ller space}, J. Reine
  Angew. Math. \textbf{473} (1996), 121--136. \MR{1390685}

\bibitem{10.2307/2037802}
David Mumford, \emph{A remark on {M}ahler's compactness theorem}, Proc. Amer.
  Math. Soc. \textbf{28} (1971), 289--294. \MR{276410}

\bibitem{parkkonen2015hyperbolic}
Jouni Parkkonen and Fr\'{e}d\'{e}ric Paulin, \emph{On the hyperbolic orbital
  counting problem in conjugacy classes}, Math. Z. \textbf{279} (2015),
  no.~3-4, 1175--1196. \MR{3318265}

\bibitem{thurston1988}
William~P. Thurston, \emph{On the geometry and dynamics of diffeomorphisms of
  surfaces}, Bull. Amer. Math. Soc. (N.S.) \textbf{19} (1988), no.~2, 417--431.
  \MR{956596}

\end{thebibliography}
\bibliographystyle{amsplain}
\end{document}